\documentclass{article}
\PassOptionsToPackage{numbers}{natbib}
\usepackage[preprint]{nips_2018}

\usepackage{theorem,ifthen,algorithm,algorithmic}
\usepackage{amssymb,amsfonts,amsmath,latexsym,dsfont}
\usepackage{tikz,pgflibraryplotmarks}
\usepackage{fullpage}
\usepackage{graphicx}
\usepackage{mathrsfs}
\usepackage{boxedminipage}

\usepackage{pgf,pgfarrows} %% pdf-drawing package
\usepackage{subfigure} %% pdf-drawing package

\usepackage[utf8]{inputenc} % allow utf-8 input
\usepackage[T1]{fontenc}    % use 8-bit T1 fonts
\usepackage{url}            % simple URL typesetting
\usepackage{booktabs}       % professional-quality tables
\usepackage{amsfonts}       % blackboard math symbols
\usepackage{nicefrac}       % compact symbols for 1/2, etc.
\usepackage{microtype}      % microtypography
\usepackage{theorem,ifthen,algorithm,algorithmic}
\usepackage{amssymb,amsfonts,amsmath,latexsym,dsfont}
\usepackage{tikz,pgflibraryplotmarks}
\usepackage{fullpage}
\usepackage{graphicx}
\usepackage{mathrsfs}

\graphicspath{{figLS/}{figMNIST/}}

\usepackage{boxedminipage}

\usepackage{pgf,pgfarrows} %% pdf-drawing package
\usepackage{subfigure} %% pdf-drawing package

\newcommand{\FrameboxA}[2][]{#2}
\newcommand{\Framebox}[1][]{\FrameboxA}

\newcommand{\hf}{{\frac 12}}

\newcommand{\grad}{\ensuremath{\nabla}}

\newcommand{\bfA}{{\bf A}}
\newcommand{\bfB}{{\bf B}}

\newcommand{\bfH}{{\bf H}}
\newcommand{\bfI}{{\bf I}}
\newcommand{\bfJ}{{\bf J}}

\newcommand{\bfQ}{{\bf Q}}

\newcommand{\bfW}{{\bf W}}
\newcommand{\bfX}{{\bf X}}

\newcommand{\bfb}{{\bf b}}
\newcommand{\bfc}{{\bf c}}

\newcommand{\bfx}{{\bf x}}

\newcommand{\bfd}{{\bf d}}

\newcommand{\bft}{{\bf t}}
\newcommand{\bfg}{{\bf g}}

\newcommand{\Ex}{{\mathbb E }} 

\newcommand{\Cov}{{\mathbb C {\rm ov}}}

\newcommand{\bfOmega}{{\boldsymbol \Omega}}
\newcommand{\bftheta}{{\boldsymbol \theta}}

\newcommand{\bfomega}{{\boldsymbol \omega}}

\newcommand{\R}{\ensuremath{\mathds{R}}}

\usepackage{xcolor}

\newtheorem{theorem}{Theorem}
\newtheorem{proof}{Proof}
\newtheorem{example}{Example}

\title{Never  look back - A modified EnKF method and its application to the training of neural networks without back propagation}

\author{
  Eldad Haber\thanks{Department of Earth and Ocean Science, The University of British Columbia, Vancouver, BC, Canada}$^{\;}$\thanks{Xtract Inc., Vancouver, Canada}\;,  Felix Lucka\thanks{Computational Imaging, Centrum Wiskunde \& Informatica (CWI), The Netherlands; Department of Computer Science, University College London, UK}\; and Lars Ruthotto\thanks{Department of Mathematics and Computer Science, Emory University, USA}$^{\;}$\footnotemark[2]
 }

\begin{document}

\maketitle

\begin{abstract}
In this work, we present a new derivative-free optimization method and investigate its use for training neural networks.
Our method is motivated by the Ensemble Kalman Filter (EnKF), which has been used successfully for solving optimization problems that involve large-scale, highly nonlinear dynamical systems.  A key benefit of the EnKF method is that it requires only the evaluation of the forward propagation but not its derivatives. Hence, in the context of neural networks it alleviates the need for back propagation and reduces the memory consumption dramatically. However, the method is not a pure "black-box" global optimization heuristic as it efficiently utilizes the structure of typical learning problems.   
Promising first results of the EnKF for training deep neural networks have been presented recently by Kovachki and Stuart. 
We propose an important modification of the EnKF that enables us to prove convergence of our method to the minimizer of a strongly convex function. Our method also bears similarity with implicit filtering and we demonstrate its potential for minimizing highly oscillatory functions using a simple example. Further, we provide numerical examples that demonstrate the potential of our method for training deep neural networks.  
\end{abstract}

\section{Introduction}

Advances in numerical algorithms for training neural networks have been fueling the recent deep learning revolution. 
Virtually all training algorithms employ derivative-based optimization methods, and in particular, stochastic gradient descent (SGD) methods  (see, e.g., \cite{Bertsekas2011,Bottou2012,Zhang-et-al-2018,deepLearning2012,GoodfellowEtAl2016,KeskarEtAl2016,LinderothShapiroWright2006,
JuditskyLanNemirovskiShapiro2009,shapiroBook} and reference within).

A key requirement of derivative-based optimization methods is an efficient way of computing the gradient of the objective function, which is used to update the parameters to be optimized. 
When training neural networks, this is commonly done using back propagation~\cite{williams1986learning}, i.e., propagating prediction errors from the output layer backwards through the  network. 
Back propagation relies upon a means to store or quickly recompute hidden features and activations, which can be prohibitive when the neural network is deep or when the size of the data is large, e.g., when training a network to learn from videos or 3D datasets.
While the storage problem is less pronounced when using stochastic gradient methods, the performance of these methods depends in a non-trivial way on parameters like batch size and learning rate~\cite{GoodfellowEtAl2016}, which often need to be tuned manually by solving the problem repeatedly.  
Also, SGD is very difficult to parallelize efficiently.

To overcome some of the drawbacks associated with derivative-based learning algorithms,  we propose a new derivative-free optimization method that can be seen as a slightly modified version of the Ensemble Kalman Filter (EnKF). 
Our method can either augment or replace the SGD for the training of neural networks. 
EnKF has a long track record in data assimilation and been applied to large-scale optimization problems, e.g., in weather and flow prediction~\cite{FahimuddinPhd,NennaEtAl2011,ott03a,hun07a,hou01a,eve03a,eve00a}.
In recent years, a version of the EnKF method that recovers an unknown parameter from noisy data has been applied to inverse problems~\cite{IgLaSt13,Ig15,SchillingsStuart2017}.
Key advantages of the EnKF method  for large-scale problems is that it can be easily parallelized and is derivative-free, i.e., it requires only evaluations of the forward operator. Another advantage of EnKF, that is particularly attractive for highly nonlinear forward problems is its potential to explore the objective function. However, EnKF is not a "black-box" metaheuristic for global optimization of arbitrary functions, like, e.g., genetic algorithms, which were also tested for deep learning recently \cite{2017arXiv171206567P}. It makes explicit and efficient use of the specific structure of learning problems. 

In this paper, we continue the work of Kovachki and Stuart  for training neural networks~\cite{KovachkiStuart2018, KovachkiStuart2018slides}, which is based on the EnKF presented in \cite{SchillingsStuart2017}. 
We modify this scheme and prove that our method converges to the global minimizer of a strongly convex function at a sublinear rate.
This is an improvement over the method discussed in~\cite{SchillingsStuart2017}, which converges to the projection of the minimizer onto a low-dimensional subspace.
Furthermore, our method can be parallelized easily, which  can be attractive in a high-performance computing environment.
Using an intuitive heuristic argument and a simple numerical experiment we also show how avoiding back propagation can be advantageous for problems that exhibit local high frequency oscillations. 
These properties motivate us to use our method for training deep neural networks even though we are not able to prove convergence to the global solution of such non-convex problems.
Using derivative-free optimization methods for training deep neural networks has the advantage that no back propagation is needed and therefore, it is possible to work with arbitrary long networks without any memory limitations.
We illustrate the potential of our method using a classical convolutional neural network for solving the MNIST problem of classifying hand-written digits~\cite{LeCunEtAl1990} and show that our method converges faster than a SGD type method.

Our paper is organized as follows. 
In Sec.~\ref{sec:enkf}, we derive our method by modifying the version of the Ensemble Kalman Filter presented in~\cite{SchillingsStuart2017}.
In Sec.~\ref{sec:conv}, we state and prove our convergence result for strongly convex functions and discuss the relation of our method to SGD. 
In Sec.~\ref{sec:imp}, we present several practical improvements of our method. 
In Sec.~\ref{numex}, we present preliminary numerical results for strongly convex problems and the training of some common neural networks.
In Sec.~\ref{sec:conc} we provide some discussion and concluding remarks.

\section{The Ensemble Kalman Filter}
\label{sec:enkf}

Among the many variants of the Ensemble Kalman Filter (EnKF), we focus our attention 
on the recent version presented in \cite{SchillingsStuart2017}. This version simplifies much of the original EnKF discussion and allows for simple analysis and explanation. The method is an iterative method for the minimization of a function of the form
\begin{eqnarray}
\label{obj}
\phi(\bftheta) = {\cal D}(F(\bftheta)),
\end{eqnarray}
where $\bftheta \in \R^n$ denotes the parameter to be recovered, 
$F:\R^n \rightarrow \R^m$ is the forward operator, and ${\cal D}:\R^m \rightarrow \R$ is a loss or misfit function.
This formulation includes, e.g., linear least-squares problems that are obtained by choosing ${\cal D}(\bft) = \hf \|\bft \|^2$ and $F(\bftheta) = \bfA \bftheta - \bfb$ for fixed $\bfA \in \R^{m\times n}$ and $\bfb\in\R^m$.
For most classification problems, $F(\bftheta)$ is a neural network and ${\cal D}(\cdot)$
is the soft-max function. The key advantage of EnKF is that it does not rely on derivatives of $F$ with respect to $\theta$, which is usually computationally demanding, one only needs to evaluate the derivative of the loss function, which is computationally cheap. We will therefore call it "derivative-free" and use the short notation $\grad {\cal D} (\theta)$ for $\grad_t{\cal D} (t)$ evaluated at $t = F(\theta)$, but again stress that this is a rather different approach to other "derivative-free" optimization heuristics that try to minimize $\phi(\bftheta)$ without utilizing any of its structure, i.e., by evaluations of $\phi$ for given $\theta$'s only.  

The first and simplest version of our algorithm consists of three steps that are listed in Algorithm~\ref{alg1}.
In step one of each iteration $j$ we randomly choose $k$ different particles $\bftheta_i$, $i=1,\ldots,k$ i.i.d. as $\bftheta \sim \bar\bftheta_j + \bfomega$. The distribution of $\bfomega$ needs fulfill $\Ex \bfomega = 0$, $\Cov(\bftheta) = \sigma^2 \bfI$ (where $\bfI$ denotes the identity matrix), and any higher moments up to forth order need to exist. Possible examples are given by $\bfomega \sim {\cal N}(0, \sigma^2 \bfI)$, or $\bfomega = \pm \sigma$ with equal probability (Rademacher distribution).
The second step requires one to evaluate the forward operator (but not its derivatives) $k$ times. 
In the third step of the algorithm, one needs to compute a step, by using a matrix, $\bfH_j$ and computing its inverse times a residual vector. The matrix can be chosen as an identity matrix  or
as $\bfH_j = \bfQ_j^{\top} \bfQ_j + {\boldsymbol \Gamma}$, 
where $ {\boldsymbol \Gamma}$ is a data Covariance matrix. The latter leads to the usual Kalman Gain matrix  \cite{eve00a,IgLaSt13,Ig15,SchillingsStuart2017}. To update the solution, we choose the step size $\mu_j$ using a simple Armijo line search to ensure reduction of the objective function. We provide some justification for this in Sec.~\ref{sec:connection}.
\begin{algorithm}[t]
\begin{algorithmic}
	\STATE \textbf{Inputs:} Starting guess $\bar\bftheta_0\in\R^n$, parameter $\sigma>0$, number of particles $k$, covariance matrix ${\boldsymbol \Gamma}$, and maximum number of iterations $N$.
\FOR{$j=0,1,\ldots,N-1$}
\STATE[1] Draw $k$ particles i.i.d. as $\bftheta_i =  \bar \bftheta_j + \bfomega_i$, $i=1,\ldots,k$, with $\mathbb{E}(\bfomega)=0$ and ${\rm Cov}(\bfomega)=\sigma^2\bfI$. Set $\bfOmega_j = [\bfomega_1,\ldots,\bfomega_k] \in \R^{n\times k}$.
\STATE[2] Evaluate the forward operator $F$ on each particle $\bftheta_i$ and on $\bftheta_j$ to construct
\begin{eqnarray}
\label{eqQ}
\bfQ_j = F(\bar\bftheta_j + \bfOmega_j)  - F(\bar\bftheta_j) := [F(\bftheta_1)  - F(\bar\bftheta_j),\ldots,F(\bftheta_k)  - F(\bar\bftheta_j)] \in \R^{m\times k}
\end{eqnarray}
%\STATE Update particles $ \bfTheta_{j+1} = \bfTheta_j - \mu_j \bfOmega_j \bfQ_j^{\top}$
\STATE[3] Set $\bfH_j=\bfI$ or $\bfH_j = \bfQ_j^{\top} \bfQ_j + {\boldsymbol \Gamma}$ and update
\begin{eqnarray}
\label{eqT}
  \bar\bftheta_{j+1} = \bar\bftheta_j - \mu_j \bfOmega_j \bfH_j^{-1} \bfQ_j^{\top}
\grad {\cal D}(\bar \bftheta_j),
\end{eqnarray}
where $\mu_j$ is determined using a line search and $\grad {\cal D}(\bar \bftheta_j)$ denotes the gradient of the loss function, $\nabla_t {\cal D}(t)$ evaluated at $t = F(\bar \bftheta_j)$. 
\ENDFOR
\STATE \textbf{Output:} Optimized parameters $\bar\bftheta_{N}$.
\end{algorithmic}
\caption{Derivative-free algorithm for solving $\min_{\bftheta} {\cal D}(F(\bftheta))$ \label{alg1}}
\end{algorithm}

At this stage, we point to an important modification to the EnKF as presented in
in \cite{eve00a,IgLaSt13,Ig15,SchillingsStuart2017}, which allows us to better analyze its convergence properties.
% Our algorithm is a slightly modified version of the EnKF approaches in \cite{SchillingsStuart2017} and \cite{eve00a}.
% As outlined below, our modification allows us to better understand its properties.
First, we re-sample the particles at each iteration. This allows us to use a simple convergence proof of the method to the global minimum of a strongly convex function. We suspect that this requirement can be relaxed in future work.
A second change from the original algorithm is at step 2, where the original EnKF algorithm defines 
the matrix $\bfQ_j$ as the difference between $F(\bar \bftheta_j + \Omega_j)$ and its mean over all the particles. 
Clearly, both versions are the same for a linear forward problem.
Finally, the last difference between our method to the original EnKF method is that
the original EnKF algorithm propagates all the particles while 
we  only update the mean of their distribution.
It is also possible to propagate all the particles but we avoid this here for simplicity. 
As we see next, these modifications allow us to prove convergence of our algorithm and the insight we obtain this way allows us to extend the method and propose ways to further accelerate it.

\section{Convergence of our algorithm and comparison to existing methods}
\label{sec:conv}
We first prove that when applied to strongly convex problems our method converges to the global minimizer. In Sec.~\ref{sec:connection} we compare our method to stochastic gradient methods and other methods commonly used in derivative-free optimization. 

\subsection{Convergence analysis for strongly convex problems}

To prove convergence of the method we need some assumptions that
are typically used also for the proof of convergence for stochastic gradient descent methods.
We prove convergence for the case $\bfH_j=\bfI$ and note that it can be extended for any symmetric positive definite $\bfH_j$.

\begin{itemize}
\item[A1:] The function $\phi(\bftheta) = {\cal D}(F(\bftheta))$ is smooth and strongly convex, i.e., there exists a $L>0$ such that for all $\bftheta_1,\bftheta_2 \in \R^n$
$$ {\cal D}(F(\bftheta_2)) - {\cal D}(F(\bftheta_1)) \ge (\bftheta_2 - \bftheta_1)^{\top} \bfJ(\bftheta_1)^{\top} \grad {\cal D}(\bftheta_1) + {\frac L2} \|\bftheta_2 - \bftheta_1\|^2. $$
Here, $\bfJ(\bftheta) = \grad_{\bftheta}F(\bftheta)^{\top} \in\R^{m\times n}$ is the Jacobian of the mapping of $F$.
\item[A2:] 
The Taylor expansion error of $F$ can be made sufficiently small, i.e.,
assuming that we sampled an $\bfomega$ with $\| \bfomega \| = {\cal O}(\sigma)$, we have that
$$ F(\bar\bftheta + \bfomega) - F(\bar\bftheta) = \bfJ(\bar \bftheta) \bfomega + 
{\cal O}(\sigma^2). $$
%% \LRnote{seems to be a smoothness assumption on $F$, right?}.
%\item[A3:] The Hessian of $\phi(\bftheta)$ is bounded, i.e., 
%$$ (\grad {\cal D}(\bar \bftheta))^{\top} \bfJ(\bar \bftheta ) \bfJ(\bar \bftheta )^{\top} \grad {\cal D}(\bar \bftheta)  $$
\end{itemize}
%
%
%% \LRnote{check if this is supposed to be $\bar{\bftheta_j}$}
%\begin{equation}\label{gj}
% \bfg_j = \bfOmega_j \bfOmega_j^{\top}\bfJ_j^{\top}\grad {\cal D}( \bar\bftheta_j).
%\end{equation}

\bigskip

Using A2 we can write the matrix $\bfQ_j$ in~\eqref{eqQ}  as
$$ \bfQ_j = \bfJ_j \bfOmega_j + {\cal O}(\sigma^2).$$
Plugging this into \eqref{eqT} we obtain
\begin{eqnarray}
\label{EnkFupdateMean2}
  \bar\bftheta_{j+1} = \bar\bftheta_j - \mu_j \left( \bfOmega_j \bfOmega_j^{\top}\bfJ_j^{\top}
\grad {\cal D}(\bar \bftheta_j) + {\cal O} (\sigma^2) \right) .
\end{eqnarray}
From our assumptions on the distribution of the particles, it easily follows that $\Ex \, \bfOmega_j \bfOmega_j^{\top} = k \sigma^2 \bfI$. If we define $\bfg_j  := \bfOmega_j \bfOmega_j^{\top}\bfJ_j^{\top} 
\grad {\cal D}(\bar \bftheta_j)$, we can compute that
\begin{eqnarray}
\label{Exg}
 \Ex \, \bfg_j = k \sigma^2 \bfJ_j^{\top}\grad {\cal D}( \bar\bftheta_j) = k \sigma^2 \grad \phi(\bar\bftheta_j), \qquad  \Ex \, \| \bfg_j \|^2  = \| \grad \phi(\bar\bftheta_j) \|^2 + {\cal O} (k^2 \sigma^4) .
 \end{eqnarray}
Now, let $\bftheta^*$ be the minimizer of \eqref{obj}. Using A1  we have that
\begin{eqnarray*}
\phi(\bftheta^*) - \phi(\bar\bftheta_j) &\ge& (\bftheta^* - \bar\bftheta_j)^{\top} \bfJ_j^{\top}\grad {\cal D}( \bar\bftheta_j) + {\frac L2} \|\bar\bftheta_j - \bftheta^*\|^2 \\
\phi(\bar\bftheta_j) - \phi(\bftheta^*) &\ge& (\bar\bftheta_j - \bftheta^*)^{\top} \bfJ_*^{\top}\grad {\cal D}( \bftheta_*) + {\frac L2} \|\bar\bftheta_j - \bftheta^*\|^2 = {\frac L2} \|\bar\bftheta_j - \bftheta^*\|^2,
\end{eqnarray*}  
where $\bfJ_* := \bfJ(\bftheta^*)$, and we used that the gradient at $\bftheta^*$  vanishes. Summing both inequalities  we obtain
\begin{eqnarray}
\label{s1}
(\bar\bftheta_j - \bftheta^*)^{\top} \bfJ_j^{\top}\grad {\cal D}( \bar\bftheta_j) \ge L \|\bar\bftheta_j - \bftheta^*\|^2.
\end{eqnarray}

Let us assume that within Algorithm~\ref{alg1}, we arrived at a concrete $\bar \bftheta_{j}$. We now derive an inequality that describes how the distance $\|\bar \bftheta_{j+1} -  \bftheta^*\|^2 $ compares to  $\| \bar\bftheta_{j} -\bftheta^* \|^2 $ \textit{in expectation}. Technically, this means that all expected values here are conditioned on $\bar \bftheta_{j}$ and only account for the randomness induced by drawing a new set of particles, which is manifest in $\bfOmega_j$. We first use \eqref{EnkFupdateMean2} and \eqref{Exg} to estimate:
\begin{align*}
 \Ex \|\bar\bftheta_{j+1} - \bftheta^*\|^2 &=
\Ex \|\bar\bftheta_{j} - \mu_j (\bfg_j + {\cal O} (\sigma^2)) - \bftheta^*\|^2 \\
&= \Ex \, \|\bar\bftheta_{j}  - \bftheta^*\|^2 - 2 \mu_j \Ex \, \bfg_j^{\top}
(\bar\bftheta_{j} - \bftheta^*) + {\cal O} (\mu_j \sigma^2) + {\cal O} (\mu_j^2  k^2 \sigma^4) \\
& \le \Ex\|\bar\bftheta_{j}  - \bftheta^*\|^2 - 
2 \mu_j k \sigma^2 \grad \phi(\bar\bftheta_j)^{\top}
(\bar\bftheta_{j} - \bftheta^*) + {\cal O} (\mu_j \sigma^2) + {\cal O} (\mu_j^2  k^2 \sigma^4)
\end{align*}
Using \eqref{s1} we obtain
\begin{align}
\label{Inequality}
 \Ex \|\bar\bftheta_{j+1} - \bftheta^*\|^2 \le 
&= (1-2 \mu_j L k \sigma^2 )\Ex  \|\bar\bftheta_{j} - \bftheta^*\|^2 + {\cal O} (\mu_j \sigma^2) + {\cal O} (\mu_j^2  k^2 \sigma^4) 
\end{align}

We can now prove convergence by induction when we choose the learning rate
$\mu_j = \frac{1}{jL k \sigma^2}$.
\begin{theorem}
	\label{theorem1}
	At the $j$-th iteration, it holds that
\begin{eqnarray}
\label{conv}
 \Ex\|\bar\bftheta_{j} - \bftheta^*\|^2 \le \frac {C}{j}, 
\end{eqnarray}
where the constant $C$ only depends on the initial distance, $\| \bftheta_0 - \bftheta^* \|$, the smoothness of $\phi$, $L$, and $k$.
\end{theorem}

\begin{proof}
For $j= 0$, it is clear that as we pick $\bftheta_0$ manually, $\Ex \| \bftheta_0 - \bftheta^* \| = \| \bftheta_0 - \bftheta^* \|$ can be bounded in the above way. The terms $ {\cal O} (\mu_j \sigma^2) + {\cal O} (\mu_j^2  k^2 \sigma^4) $ in \eqref{Inequality} all come from smoothness assumptions on $\phi$ and can thus be bounded accordingly. If we choose the learning rate $\mu_j = \frac{1}{jL k \sigma^2}$ and absorb remaining dependencies on $k$ and $L$ in $C$ as well, we obtain
\begin{align}
\label{Inequality2}
 \Ex \|\bar\bftheta_{j+1} - \bftheta^*\|^2 \le 
&= \left(1- \frac{2}{j} \right)\Ex  \|\bar\bftheta_{j} - \bftheta^*\|^2 + \frac{C}{j} + \frac{C}{j^2} .
\end{align}
If we then assume that the convergence holds at iteration $j$, we prove that it holds for $j+1$: 
\begin{align*}
  \Ex\|\bftheta_{j+1} - \bftheta^*\|^2 &\le  
 \left(1- \frac{2}{j} \right) \frac{C}{j}  + \frac{C}{j} + \frac{C}{j^2} = C \left( \frac{1}{j} - \frac{2}{j^2} + \frac{1}{j} + \frac{1}{j^2}\right) \le  {\frac {C}{j+1}}.
\end{align*}
\end{proof}

\subsection{Comparison to other methods}
\label{sec:connection}
It is interesting to note the difference between our derivative-free optimization algorithm  to stochastic gradient descent (SGD) and its variants. To this end, note that the SGD iteration can be written as
$$ \bftheta_{j+1} = \bftheta_j - \eta_j 
\bfJ_j^{\top} \bfX_j \bfX_j^{\top} \grad {\cal D}(\bftheta_j), $$
where the rows of $\bfX_j\in\R^{s\times m}$ are randomly chosen rows of the $m\times m$ identity matrix and $\eta_j$ is the learning rate.
Although the choice of the learning rate in Theorem~\ref{theorem1} is similar to the one commonly used to prove convergence of SGD, it is important to note one major difference between our update rule and the SGD step.
In contrast to SGD step, the step size in our method can always be chosen so that the objective function is non-increasing. This is because the computed direction is a descent direction over the subspace spanned by $\bfOmega_j$, since
$$ -\grad \phi^{\top} \bfOmega_j \bfOmega_j^{\top} \bfJ_j^{\top} \grad {\cal D} 
= - (\grad {\cal D})^{\top} \bfJ_j^{\top}\bfOmega_j \bfOmega_j^{\top} \bfJ_j^{\top} \grad {\cal D} \le 0. $$
 Hence, common line search strategies can be used to ensure that the objective is  monotonically non-increasing. As we see in the numerical examples this makes the selection of hyper parameters easier than in the case of SGD.
 
 Of course this property does not come without a price. This version of the method requires the propagation of all the data making it expensive for problems where the number of data is very large. Nonetheless, it is possible to combine the method with a stochastic selection of the data making it competitive with standard SGD iteration.
 In this case, the iteration has the form
\begin{eqnarray}
\label{EnGFSGD}
 \bar\bftheta_{j+1} = \bar\bftheta_j - \mu_j 
\bfOmega_j \bfH_j^{-1}\bfQ_j^{\top} \bfX_j \bfX_j^{\top}  \grad {\cal D}(\bftheta_j).
\end{eqnarray}
It is straight forward to extend the proof of convergence for this case when $\bfX_j$ and $\bfOmega_j$ are uncorrelated.

Although the convergence proof covers only strongly convex problems, one motivation to use our derivative-free method for highly nonlinear problems can be made by looking at the behavior of the method for problems with small, yet, high-frequency oscillations. 
\begin{example}\label{exam:oscillatory}
Let $l\in\R$, $\epsilon>0$, $\bfA \in \R^{m\times n}$, and $\bfB \in \R^{m\times n}$ be given and
$$ F(\bftheta) = S(\bftheta) + \epsilon N(\bftheta) = \bfA \bftheta + \epsilon \sin(l \bfB \bftheta) $$
where $S$ is a smooth function and $N$ is an oscillatory one,
and assume the least squares misfit function, ${\cal D}(\bft) =\hf\|\bft\|^2$.
In this case the necessary conditions for a minimum are
$$\bfJ^{\top} \grad {\cal D}(\bftheta) =  (\bfA^{\top} + \epsilon l\bfB^{\top} {\rm diag} (\cos(l \bfB \bftheta)) )  
(\bfA \bftheta + \epsilon \sin(l \bfB\bftheta)  - \bfd^{\rm obs}) = 0 $$
where $\bfJ = \bfA + \epsilon l {\rm diag} (\cos(l \bfB \bftheta))\bfB $.
Clearly, the misfit function can have many local minima. 
 For large values of $l$, the function presents some local highly oscillatory behavior as well as some ``slow'' low-frequency modes. As a result, if we are to compute the solution using a derivative-based technique, we would probably fail.

It is  important to understand the cause of the oscillatory gradients. Note that the gradient of our function is simply
$$ \grad_{\bftheta} {\cal D} = \bfJ^{\top} \grad {\cal D} = 
  (\grad S)^{\top} {\cal D} + \epsilon (\grad N)^{\top}\grad {\cal D}.$$
  While $\grad {\cal D}$ is smooth or oscillates mildly, the second part, $\bfJ$ contains much higher oscillations.
  For the problem at hand, $\grad S$ is linear, however, the oscillations in $\grad N$ are amplified by a factor of $l$. Indeed, even if the function has small oscillations, its derivative has much larger oscillations and therefore, the solution can be difficult to obtain.
The source of the oscillations can be tracked to the gradient of the oscillatory part $N$. 

It is important to note that the approximate Jacobian $\bfQ_j =  F(\bar\bftheta_j+\bfOmega_j ) - F(\bar\bftheta)$, computed in our method does not suffer from this problem as long as we choose the elements of $\bfOmega$ sufficiently large. This is because we only evaluate the forward operator and determining the step is similar to numerical differentiation.  Thus, our method may present better properties for problems with highly oscillatory local behavior compared with derivative-based (i.e. back propagation) descent methods.  

 \end{example}

The above discussion shows that there is a strong link between the EnKF method and implicit filtering discussed in \cite{kelley3}. In fact, the EnKF can be seen as a stochastic version of implicit filtering, that is known to work well for noisy functions.
Finally, the method can be seen as a particular implementation of a stochastic coordinate descent, where the coordinates are picked by the matrix $\bfOmega_j$ and the gradient is evaluated numerically.

\section{Practical improvements and extensions} 
\label{sec:imp}

While Algorithm~\ref{alg1} can be used directly for the solution of the training problem, some simple modifications can boost its performance considerably.

\subsection{Reducing the number of particles per iteration}

The computational cost at each iteration scales linearly with the number
of particles used and therefore, it is desirable to reduce the number of particles used at each iteration. 
To this end, we propose to re-use the matrices computed from previous particles at earlier iterations at the computation.
The modified algorithm is described in Algorithm~\ref{alg2}. 
\begin{algorithm}[t]
\begin{algorithmic}
\STATE \textbf{Inputs:} Starting guess $\bar\bftheta_0\in\R^n$, parameter $\sigma>0$, number of particles $k$, covariance matrix ${\boldsymbol \Gamma}$, and maximum number of iterations $N$. Allocate $\bfOmega=[]$ and $\bfQ = []$
\FOR{$j=0,1,\ldots,N-1$}
\STATE[1] Draw $k$ particles i.i.d. as $\bftheta_i =  \bar \bftheta_j + \bfomega_i$, $i=1,\ldots,k$, with $\mathbb{E}(\bfomega)=0$ and ${\rm Cov}(\bfomega)=\sigma^2\bfI$. Set $\bfOmega_j = [\bfomega_1,\ldots,\bfomega_k] \in \R^{n\times k}$.

\STATE[2] Evaluate the forward operator $F$ on each particle $\bftheta_i$ and on $\bftheta_j$ to construct
\begin{eqnarray}
\label{eqQMem}
\bfQ_j = F(\bar\bftheta_j + \bfOmega_j)  - F(\bar\bftheta_j) := [F(\bftheta_1)  - F(\bar\bftheta_j),\ldots,F(\bftheta_k)  - F(\bar\bftheta_j)] \in \R^{m\times k}
\end{eqnarray}
\STATE[3] Update matrices $\bfOmega \leftarrow [\bfOmega, \bfOmega_j]$ and
$\bfQ \leftarrow [\bfQ, \bfQ_j]$ by adding columns.
\STATE[4] Set $\bfH_j=\bfI$ or $\bfH_j = \bfQ^{\top} \bfQ + {\boldsymbol \Gamma}$ and update
\begin{eqnarray}
\label{eqTMem}
  \bar\bftheta_{j+1} = \bar\bftheta_j - \mu_j \bfOmega \bfH_j^{-1} \bfQ^{\top}
\grad {\cal D}(\bar \bftheta_j)
\end{eqnarray}
where $\mu_j$ is determined using a line search and $\grad {\cal D}(\bar \bftheta_j)$ denotes the gradient of the loss function, $\nabla_t {\cal D}(t)$ evaluated at $t = F(\bar \bftheta_j)$.
\STATE \textbf{Output:} Optimized parameters $\bar\bftheta_{N}$.
\ENDFOR
\end{algorithmic}
\caption{Derivative-free algorithm for solving $\min_{\bftheta} {\cal D}(F(\bftheta))$ with memory\label{alg2}}
\end{algorithm}
This version of the algorithm is similar to the variance reduction SGD method
proposed in \cite{2012arXiv1202.6258L}. 
In practice, one allocates memory for, say $\ell$ vectors for $\bfQ$ and $\bfOmega$ and saves only the last $\ell$ vectors. It is important to see that this version of the method requires much less forward calculations at each step. We will show how this version of the code can be advantageous compared to the ``vanilla'' version presented in Algorithm~\ref{alg1}. 

\subsection{A Gauss-Newton like update}

The update of $\bar \bftheta$ proposed in \eqref{eqT} or \eqref{eqTMem} reduces to a steepest descent algorithm when we use a full basis to sample the problem.
However, the original EnKF algorithm uses a Kalman gain matrix to update the solution. It is well known that Kalman filtering and the Kalman gain can be interpreted as a Gauss-Newton iteration \cite{vd}. 
We can thus use the matrix $\bfQ$ to assess an approximation of the Hessian. Noting that
$$ \Ex\, \bfQ \bfQ^{\top} \approx \Ex\,  \bfJ \bfOmega\bfOmega^{\top} \bfJ^{\top}  = \sigma^2 k \bfJ  \bfJ^{\top} $$,
one can use the update
\begin{eqnarray}
\label{eqT2}
  \bar\bftheta_{j+1} = \bar\bftheta_j - \frac{\mu_j}{\sigma^2 k} \bfOmega \bfQ^{\top} (\bfQ \bfQ^{\top} + {\boldsymbol \Gamma})^{-1} 
\grad {\cal D}(\bar \bftheta_j)
\end{eqnarray}
where $\boldsymbol \Gamma$ is a data covariance matrix.
This update is similar to the common EnKF presented in the original work
\cite{eve00a}.
It is important to realize that the system \eqref{eqT2} needs not be formed and it is possible to use $k$ steps of the conjugate gradient method, preconditioned by $\boldsymbol \Gamma$ to solve the system exactly.

\section{Numerical examples}
\label{numex}

In this section, we present results from two numerical experiments. 
First, we demonstrate our algorithm's potential to converge on non-smooth objective functions. 
Second, we compare our method to a SGD scheme for the MNIST problem of classifying hand-written digits.

\subsection{Application to nonlinear regression}

We apply our method to regression problem described in Example~\ref{exam:oscillatory}, which involves the highly oscillatory forward problem
% a very simple problem that demonstrates the properties of our algorithm as well as some of the potential pitfalls.
\begin{eqnarray}
\label{f1}
F(\bftheta) = \bfA \bftheta + \epsilon \sin(l\bfB \bftheta).
\end{eqnarray}
Here,  we choose $\bfA$ and $\bfB$ as random matrices of size
$300\times 200$, we set $l$ to $20$, and use $\epsilon=1$. 
Figure~\ref{fig1} visualizes the resulting objective function and the results obtained using both versions of our method as listed in Algorithm~\ref{alg1} and Algorithm~\ref{alg2} and varying the number of particles $k \in \{5, 10, 15, 20, 25\}$. 
As can be seen from the plot of the objective function, this problem is challenging for gradient-based methods.
Similar to implicit filtering methods~\cite{kelley3}, our method is better-equipped to deal with this non-smoothness because of the spatial sampling when approximating the gradient. 
We note that the convergence accelerates when more particles are used and the benefit of storing intermediate evaluations of the forward model is obvious.
\begin{figure}
\begin{center}
\includegraphics[width=\textwidth]{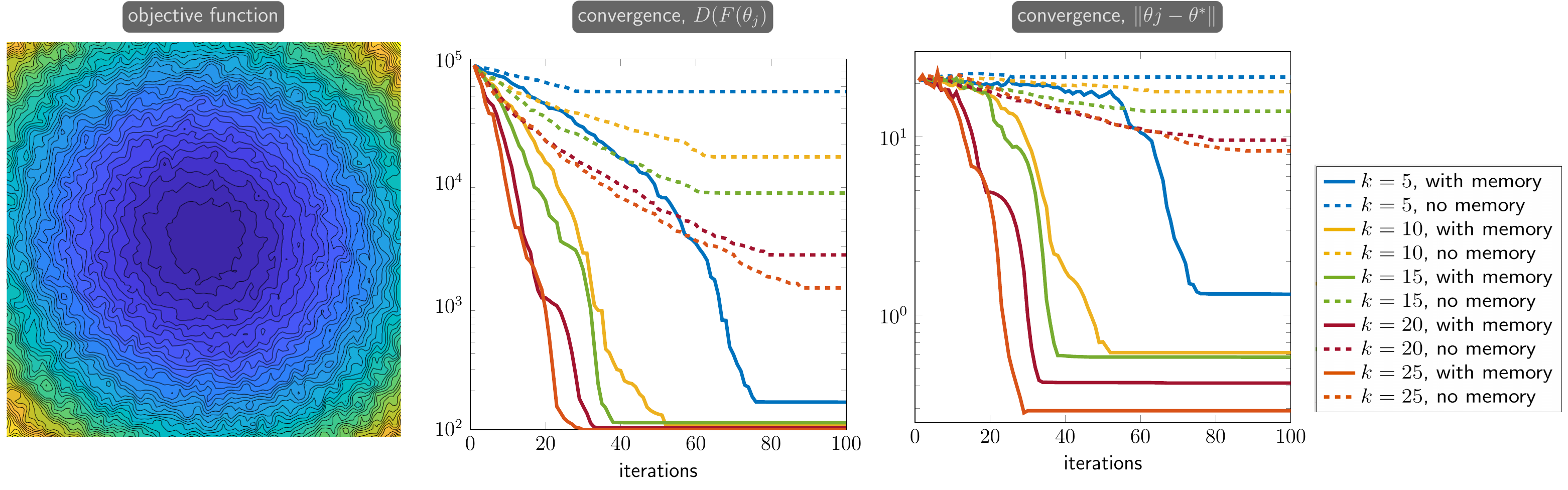}
\caption{Convergence of our methods for a regression problem with highly oscillatory forward operator as described in Example~\ref{exam:oscillatory}. Left: Plot of the objective function for $m=300, n=200, l=20,$ and $\epsilon=1$ along two randomly chosen orthogonal directions. Center: Reduction of the objective function value for different values of $k$ (indicated by colors) and different algorithms (dashed lines indicate Algorithm~\ref{alg1} and solid line represent Algorithm~\ref{alg2}). Right indicates the distance to the true solution. Note the improvements achieved by increasing $k$ and storing previous forward evaluations. 
\label{fig1}}
\end{center}
\end{figure}
\subsection{Image classification using Convolutional Neural Networks}

We use our method to train the weights of a simple neural network to classify the hand-written digits in the MNIST data set~\cite{LeCunEtAl1990}. 
The data set consists of 60,000 gray scale images of resolution $28\times28$ that is divided into 50,000 training and 10,000 test images associated with their labels.

Here, we denote by $F(\theta,\bfx)$ the forward propagation of an image $\bfx$ through the CNN consisting of two hidden layers. The first one uses a convolution with $5\times5$ stencils and  32 output channels and a ReLU activation function. This layer is followed by an average pooling operator and a second convolutional layer with $5\times5$ stencils,  64 output channels, and also ReLu activation. The last part of the forward model is  a second average pooling operator, reducing the image size to $7\times7$. The network has $52,000$ weights.

As common in classification using deep neural networks, a fully-connected layer followed by a softmax hypothesis function, $H$, is applied to the output of the neural network and the result is compared to the given label using a cross entropy loss function.  
In the notation of our method, we interpret these parts as being associated with the misfit function ${\cal D}$ but note that this function depends on the weights of the fully connected layer, denoted by $\bfW$.
Overall, denoting the image-label pairs by $\{(\bfx_i,\bfc_i) \}_{i=1}^s$ we phrase the learning problem as an optimization problem 
\begin{equation}\label{eq:optProb}
	\min_{\bfW,\bftheta} \frac{1}{s} \sum_{i=1}^s  D(H(\bfW,F(\bftheta,\bfx_i)) + R(\bfW,\bftheta),
\end{equation}
where  $R$ is a regularizer. Here, for simplicity we use a weight decay regularizer for $\bfW$ with a weight of $100$ and no regularization on $\bftheta$.

While the objective function in~\ref{eq:optProb} is non-convex in the weights $\bftheta$, it is convex with respect to $\bfW$ for the softmax cross entropy loss function and regularizer at hand. 
Hence, for a given choice of $\bftheta$ the optimal $\bfW^*(\bftheta)$ can be computed efficiently using, e.g., Newton's method, which is feasible in our case since the number of elements in $\bfW$ is only $31,370$.
We exploit this structure in our method in the following way.
In each iteration, we evaluate the forward operator $F$ for all the examples in the training data set. 
Then, we solve the classification problem using 10 iterations of an inexact Newton method using up to 20 iterations of a Conjugate Gradient (CG) scheme to determine the direction.
Keeping $\bfW^*(\theta)$ we approximate the Jacobian of the neural network by using Algorithm~\ref{alg1} with one mini batch consisting of 16 examples and four particles.
Line-search and updating of the current value of $\bftheta$ is done as described above. 

Figure~\ref{fig2} shows the convergence of our method in comparison with the stochastic gradient method ADAM with a mini batch size of 16 and a learning rate $\eta_j=10^{-3}/\sqrt{j}$.
The number of epochs is chosen so that an equal number of forward propagations is performed.
Note that the $x$-axis scales with the number of forward propagations and we do not account for back-propagation in ADAM and the classification solves in our method. 
ADAM generalizes slightly better in this case, achieving a test accuracy of (99.12\% vs. 
98.38\%) but our method converges considerably faster both in terms of forward propagations and in terms of runtime (238.5 vs. 1,751 secs on a TITAN X with CUDNN 7.1 in MATLAB2017b).

\begin{figure}[t]
	\begin{center}
		\includegraphics[width=.85\textwidth]{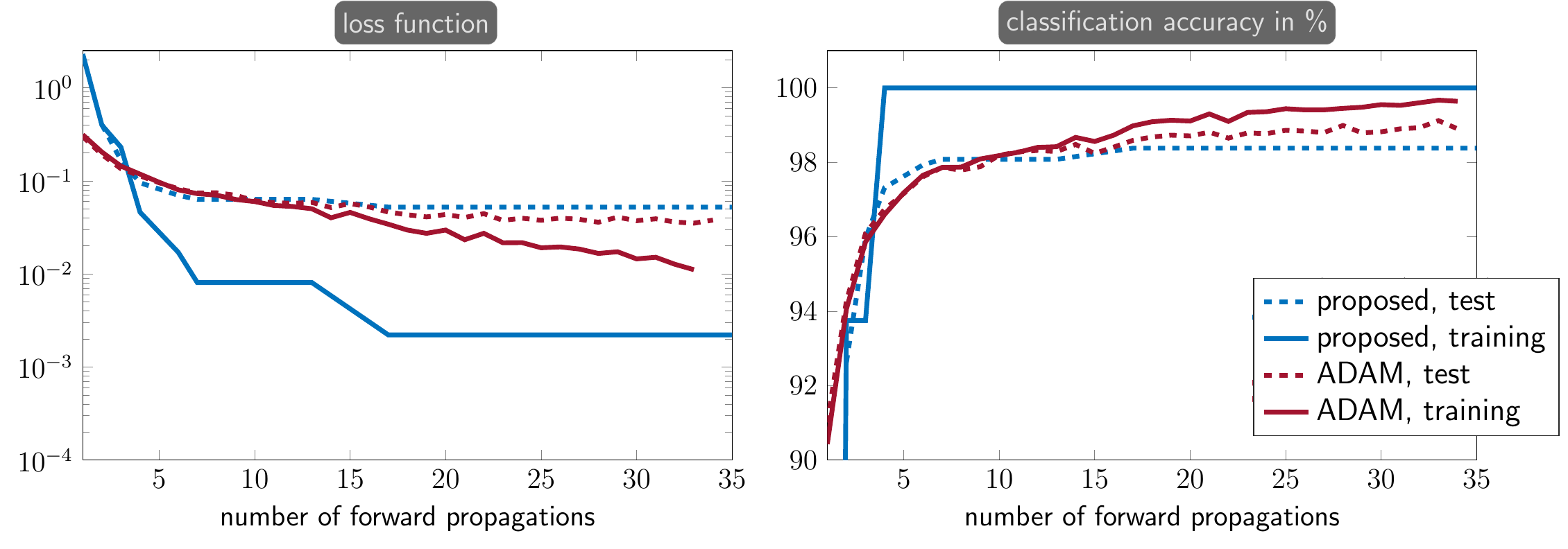}
	\end{center}
	\caption{Convergence of our derivative-free method with variable projection (blue) and the stochastic gradient method ADAM  (red) for training a convolutional neural network with 2 layers to classify the MNIST dataset. Left: Convergence of the loss function on the training (solid line) and test data (dashed line). Right: classification accuracy. The $x$-axis is scaled according to the number of forward propagations. It can be seen that  ADAM generalizes slightly better while our proposed method converges considerably faster. }
	\label{fig2}
\end{figure}

\section{Conclusions}
\label{sec:conc}
In this paper, we present a new derivative-free method for minimizing functions that are a concatenation of a misfit or loss function and a forward operator. 
Our method is derivative-free in the sense that the Jacobian of the forward operator is estimated using its evaluations at randomly chosen points around the current iterate while the derivative of the misfit is computed exactly. 
Our method can be seen as a modified version of the EnKF with improved convergence properties for strongly convex problems.
Borrowing ideas from variance reduction methods we present an improved version of our method that stores intermediate evaluations of the forward operator. 

Our numerical experiments show two benefits of our method. 
First, we use a simple nonlinear least-squares problem with an oscillatory forward operator to demonstrate the smoothing property of our method. In fact our method can be seen as a randomized version of the implicit filtering~\cite{kelley3} applied to a randomly chosen low-dimensional subspace. 
Second, motivated by the works~\cite{KovachkiStuart2018, KovachkiStuart2018slides},  we show that our method achieves comparable classification results to the SGS variant ADAM for the MNIST data set. In particular, our method parallelizes better and achieves a lower loss value within only a few forward propagations. Notable only four particles are used to optimize around 50,000 weights which motives the use of our method also for more challenging learning tasks.

\section{Acknowledgements}
We thank Nikola Kovachki and Andrew M. Stuart (Caltech) for sharing their results on applying EnKF to neural networks~\cite{KovachkiStuart2018, KovachkiStuart2018slides} that motivated this work.
LR's work is supported by the US National Science Foundation (NSF) awards DMS 1522599 and DMS 1751636. FL's is supported in parts by the European Union's Horizon 2020 research and innovation programme H2020 ICT 2016-2017 under grant agreement No 732411 (as an initiative of the Photonics Public Private Partnership) and the Netherlands Organisation for Scientific Research (NWO 613.009.106/2383).

\footnotesize
\setlength{\bibsep}{1pt}
  \bibliographystyle{plain}

   % \bibliography{biblio}
\end{document}